\documentclass[10pt]{article}
\usepackage{amssymb}
\usepackage{amsmath}
\usepackage{amsthm}
\usepackage{amsfonts}
\usepackage[ansinew]{inputenc}
\usepackage{amsthm,amsfonts}
\makeatletter \@addtoreset{equation}{section} \makeatother

\newtheorem{theorem}{Theorem}[section]
\newtheorem{corollary}{Corollary}[section]
\newtheorem{lemma}{Lemma}[section]
\newtheorem{proposition}{Proposition}[section]

\newtheorem{definition}{Definition}[section]

\newtheorem{alg}{Algorithm}

\makeatletter
\renewcommand*{\@biblabel}[1]{\hfill#1.}
\makeatother

\title{A proximal point algorithm with generalized proximal distances to BEPs}

\author{Bento, G. C.  \thanks{IME, Universidade Federal de Goi\'as, Goi\^ania, GO 74001-970, BR ({\tt glaydston@mat.ufg.br}).}
\and 
Cruz Neto, J. X. 
\thanks{CCN, DM, Universidade Federal do Piau\'i,
Terezina, PI 64049-550, BR ({\tt jxavier@ufpi.edu.br})}
\and 
Lopes, J. O. 
\thanks{CCN, DM, Universidade Federal do Piau\'i,
Terezina, PI 64049-550, BR ({\tt jurandir@ufpi.edu.br}).}
\and
Soares Jr, P.A. \thanks{CCN, DM, Universidade Etadual do Piau\'i, Terezina, PI 64002-150, BR ({\tt pedrosoares@uespi.br})}
\and
Soubeyran, A.  \thanks{Aix-Marseille University (Aix-Marseille School of Economics), CNRS \&  EHESS, FR, ({\tt antoine.soubeyran@gmail.com})}
}
\begin{document}
\maketitle

\begin{abstract}
We consider a bilevel problem involving two monotone equilibrium \mbox{bifunctions} and we show that this problem can be solved by a  \mbox{proximal} point method with generalized proximal distances. We propose a framework for the convergence analysis of the sequences generated by the algorithm. This class of problems is very interesting because it covers mathematical programs and optimization problems under equilibrium constraints. As an application, we consider the problem of the stability and change dynamics of task's allocation in a hierarchical organization.



\end{abstract}


\pagestyle{myheadings}
\thispagestyle{plain}
\section{Introduction}
 Consider the problem of the stability and change dynamics of task's allocation in a hierarchical organization. This
 is, among others, a crucial point for the dynamics of organizations in
 Economics and Management Sciences. At a higher level, there is a large literature about stays and changes dynamics, starting with the work in Economics presented by Schumpeter in~\cite{S2002a, S2002b}, and in Management Sciences by Nelson and Winter in~\cite{NW1982}, within an evolutionary perspective inspired by the theory of evolution in biology. These dynamics abound (see, for
 example, Leana and Barry~\cite{LB2000}). At the organizational level, they said \textquotedblleft \ldots organizations and individual employees  increasingly are pursuing change in how work is organized, how it is
 managed and in who is carrying it out. At the same time, there are numerous
 individual, organizational, and societal forces promoting stability in work
 and employment relations". In this article, the authors examine
 \textquotedblleft change and stability and the forces pushing individuals
 and organizations to pursue both"  and \textquotedblleft some level
 of tension between stability and change is an inevitable part of
 organizational life\ldots ". To hope to solve this very important problem
 for the survival and dynamic efficiency of organizations, the most important
 step is to embed this problem in a larger one. We give an answer to this
 dynamic task's allocation problem in the following way: we propose first a
 simple model of task's allocation in a hierarchical organization. Then, we
 use a recent Variational rationality approach presented by Soubeyran in~\cite{S2009,S2010} as a required enlarged framework which modelizes and unifies a lot of worthwhile stability and change dynamics which end in variational
 traps.
 
At the mathematical level, given a nonempty subset $K$ of $\mathbb{R}^n$ and $f,h: K \times K\to \mathbb{R}$ two bifunctions satisfying the property $f(x,x) =0,\ h(x,x)=0$, $x\in K$, the present paper considers the following bilevel equilibrium problem: 
\begin{equation}\label{BEP}
\mathbf{(BEP)}\qquad \mbox{find}~~\bar{x}\in S(f,K)~\mbox{such that}~h(\bar{x%
},y)\geq 0, \ \forall y\in S(f,K),
\end{equation}%
where $S(f,K)=\{u\in K:f(u,y)\geq 0, \forall y\in K\}$.
In Behavioral Sciences, this means that a leader of an organization can
choose, among all of them, an equilibrium $\bar{x}$ $\in $ $S(f,K)$, which
is preferred by his followers to all his other equilibria, i.e., such that $h(%
\bar{x},y)\geq 0$, $y\in S(f,K)$. This defines a hierarchical
equilibrium, a stability issue where the leader and all his followers prefer
to stay than to move (a hierarchical stability condition).

The bilevel equilibrium problem BEP has been widely studied and is a very active field of research. One of the motivations is that it covers optimization problems and mathematical programs with equilibrium constraints. These problems were addressed by Luo et al. in~\cite{LPD1996} and Migdalas et al. in~\cite{MPV1988}. Bilevel problems have first been formalized as optimization problems in the early 1970s by Bracken and McGill in~\cite{BM1973}. 

In the linear setting, some authors have presented iterative processes to approximate a solution of Bilevel problems. Cabot in~\cite{C2005} built an algorithm which is able to minimize hierarchically several functions over their successive argmin sets. 
Moudafi in~\cite{M2010} presented a proximal method for a class of monotone BEPs. More recently, Ding in \cite{D2010} used the auxiliary problem principle to BEPs. 
In this paper, under the hypothesis of monotonicity, we present a proximal algorithm with generalized distances for BEP. One of the reasons for using generalized distances instead of Euclidean norms is that the calculations and equations can be greatly simplified by an appropriate choice of a generalized distance that allows us to explore the geometry of the constraints. A broad explanation about generalized proximal distances is given by Auslender and Teboulle in~\cite{AT2006}, Burachik and Dutta~\cite{BD2010} and references therein. We point out that our algorithm retrieves and generalizes the proximal point method for bilevel equilibrium problems presented in~\cite{M2010}.

The organization of our paper is as follows. In Section 2, we present some notions of worthwhile stay and change dynamics of task's
allocation within an organization and give some elementary facts on generalized distances needed for reading this paper. In Section~3, we present basic hypotheses for equilibrium problems. In Section 4, we consider a generalized proximal distance as a regularization term. In Section 5, we consider a proximal point algorithm with generalized proximal distances to solve bilevel equilibrium problems, we derived a convergence analysis and examine the problem of the stability and change dynamics of task's allocation in a hierarchical organization. 

\section{Worthwhile stay and change dynamics of task's allocation within an organization}

Let us consider  how, in a dynamical setting, a leader of
a hierarchical firm can manage the allocation of a given set of tasks between
different workers to produce a final good whose quantity and quality can be
chosen, each period, in order to match better and better the consumer needs
to finally ends in a mature product. At this end, the entrepreneur stops to
change the allocation of tasks, because it is not worthwhile to innovate
more. This is a typical problem of task's allocation for an organization. A
large literature exists on this topic in Management Sciences. However, very
few dynamically formalized models exist. A static model is easy to give, using
a leader-follower formulation as a bilevel equilibrium problem. What is
really difficult is to give a dynamic formulation. This have been done by Bento et al. in~\cite{BCSS2014}, using a recent  Variational rationality approach of stay and change dynamics (see \cite{S2009, S2010}) which unifies a lot of
different points of views related to stability and change dynamics in
different disciplines (Psychology, Economics, Management Sciences, Political
Economy, Decision theory, Game theory, Artificial Intelligence,~\dots). 

\subsection{A static model of task's allocation in a hierarchical
organization}

For a hierarchical organization with a leader $l$ and several
followers $j\in J$, where $J$ is the set of followers, there are two polar cases: 
\begin{itemize}
\item[i)] authority: the leader $l$ chooses the collective action $x=(x^{l},x^{J})\in X$, where $X=X^l\times X^J$ is the set of pairs of actions  of the leader and the followers, where $x^l$ belonging to $X^l$ is the action of the leader  and $x^J$ belonging to $X^J$ is the profile of actions of the followers.  
\item[ii)] delegation: the leader $l$ chooses first action $x^{l}\in X^{l}$ and
the followers carry out the profile of actions $x^{J}=\left\{ x^{j}: j\in
J\right\} \in X^J$ in their interests, or best interests.
\end{itemize}
The subset of followers $J$ can be chosen or not, depending of the model. In
the first case, the size of the firm is a choice variable as well as the allocation of tasks. Then, in a dynamic setting, the leader is allowed to hire and to fire followers (other different formulations of an organization are presented by Bento and Soubeyran in~\cite{BS2014}, Bao et al. in~\cite{BMS2014}, and also in~\cite{S2009,S2010}). This contrasts with this paper where the size $J$ of
the firm is a given. The topic is only the allocation of tasks between the
leader and all the given followers as specialized skilled workers. Let $
I=\left\{ 1,2,..,i,..,n\right\}$ be the list of different tasks available
to the leader and the followers. In this context, the action $x^{j}=(x_{i}^{j},i\in I)\in \mathbb{R}_{+}^{n}$ of agent $j$ can be identified to
the vector of effort levels $x_{i}^{j}\geq 0$ this agent $j$ spends in doing
each task $i\in I$. The leader can choose the allocation of tasks $%
x=(x^{l},x^{J})\in X$ and the means $m(x)=(m^{l}(x^{l}),m^{j}(x^{j}),j\in
J $) allowed to him and to each worker to perform their different tasks.
This allows him to choose the quantity $q(x)$ and the quality $s(x)$ of this
final good. The revenue of the entrepreneur is $\varphi \left[ \mathfrak{q}%
(x),s(x)\right]$. His operational costs $\rho \left[ m(x)\right] +w^{J}(x)$
are the sum of his costs $\rho \left[ m(x)\right] $ to acquire the required
means $m(x)$, and the sum of the wages $w^{J}(x)=\Sigma _{j\in J}$ $w^{j}(x)$
paid to each employed worker $j\in J$. In this model, individual wages $%
w^{j}(x)$ depend of the \ profile $x$ of efforts of all workers (team
incentives). Individual incentives $w^{j}(x^{j}),j\in J,$ work as well (as a
variant). Then, in a given period, the profit of the entrepreneur is $%
g^{l}(x)=\varphi \left[ \mathfrak{q}(x),s(x)\right] -\rho \left[ m(x)\right]
-w^{J}(x)\in \mathbb{R}$. The net payoff of each skilled employed worker $j\in J$ is $%
g^{j}(x)=w^{j}(x)-\delta ^{j}(x^{j})$ where $\delta ^{j}(x^{j})\geq 0$ is
the disutility of effort for worker $j$. Let $g^{J}(x)=\Sigma _{j\in
J}g^{j}(x)=w^{J}(x)-\delta ^{J}(x)$ be the sum of all the net payoffs of all
different workers, where $\delta ^{J}(x)=\Sigma _{j\in J}$ $\delta
^{j}(x^{j})$ is their total disutility of efforts. The weighted payoff of
the organization is $g(x)=\varepsilon g^{l}(x)+g^{J}(x),$ where $\varepsilon
>0$ is the weight allowed to the profit of the leader. A famous example of an endogenous production function of quality is presented by Kremer in~\cite{K1993}. The
bilevel equilibrium problem defined in this paper requires authority (and full knowledge of the profile of efforts) for the leader $l$ of an
organization, because he is allowed to choose the entire vector of efforts $%
x=(x^{l},x^{J})\in X$.

\subsection{Variational rationality: how successions of worthwhile stays
and changes end in variational traps}

The VR variational rationality approach exami-nes stability and change
dynamics of human behaviors (see, \cite{S2009, S2010}). It focuses the
attention on three main concepts: worthwhile changes, worthwhile
transitions, and traps (aspiration points, stationary traps and variational
traps). The definition and modelization of these leading concepts requires
to define a list of intermediary concepts. Each of them needs lengthy
discussions for suitable applications in different disciplines. Let us
consider a general and unspecified formulation of the VR approach given in~\cite{S2009,S2010}, which allows a lot of more specific formulations. Then,
we will apply it to our specific, but important, example.

\subsubsection*{Worthwhile temporary stays and change dynamics}

The past, current and future periods are $k,k+1,k+2.$ The universal space of
actions is $X$. It represents all possible past, present and future actions
which can be discovered as time evolves. The past and current actions are $%
x^{k}\in X$ and $x=x^{k+1}\in X$. The experience of the agent at the end of
the past period $k$ is $e_{k}\in E$, where $E$ is the set of  feasible experiences
the agent can eventually acquire.
 
\begin{enumerate}
\item[(a)] \textbf{Worthwhile stay and change transition.} The  VR approach, see \cite{S2009, S2010}, modelizes a lot of behavioral
dynamics as a succession $x^{k+1}\in W_{e_{k},\xi _{k+1}}(x^{k})$, $k\in \mathbb{N}$, of
worthwhile transitions entwining temporary stays $x^{k+1}= x^{k}$ and
changes $x^{k+1}\neq x^{k},$ ending in variational traps $x^{\ast }\in X$ (to be defined below). An point $x$ of the universal space $X$ \ can be an
action (doing), or a state (having, or being). For the agent, a \ change $%
x^{k}$ $\curvearrowright x^{k+1}\in W_{e_{k},\xi _{k+1}}(x^{k})$ is
worthwhile, when his ex ante motivation to change $M_{e_{k}}(x^{k},x^{k+1})$
is sufficiently higher (more than a chosen satisficing level  $\xi _{k+1}>0$) \ than his ex
ante resistance to change, $R_{e_{k}}(x^{k},x^{k+1})$. Then, it is worthwhile to change from $x^{k}$ to $x^{k+1}$ iff 
\[
x^{k+1}\in W_{e_{k},\xi _{k+1}}(x^{k})\Longleftrightarrow
M_{e_{k}}(x^{k},x^{k+1})\geq \xi_{k+1}R_{e_{k}}(x^{k},x^{k+1}),
\]
i.e., 
\[
M_{e_{k}}(x^{k},x^{k+1})/R_{e_{k}}(x^{k},x^{k+1})\geq \xi_{k+1}, \qquad x^{k+1}\neq x^{k},
\]
where $\xi_{k+1}>0$ represents a satisficing (high enough) worthwhile to
change ratio.

Motivation and resistance to change are two complex variational concepts
which admit a lot of variants; for details see ~\cite{S2009, S2010}).
Motivation to change 
\[M_{e_{k}}(x^{k},x^{k+1})=U_{e_{k}}\left[
A_{e_{k}}(x^{k},x^{k+1})\right] 
\] 
is the utility $U_{e_{k}}\left[\cdot \right] $
of advantages to change, $A_{e_{k}}(x^{k},x^{k+1}),$ while resistance to
change 
\[
R_{e_{k}}(x^{k},x^{k+1})=D_{e_{k}}\left[ I_{e_{k}}(x^{k},x^{k+1})\right]
\]
is the disutility $D_{e_{k}}\left[\cdot \right] $ of inconvenients to
change $I_{e_{k}}(x^{k},x^{k+1}).$ A \ worthwhile change is an acceptable
change which balances satisfactions and sacrifices, improvements and costs
of improving, or desirability and feasibility. The famous satisficing principle shown by Simon in~\cite{S1955} is a specific case of satisfactions with no
sacrifices (see~\cite{S2009, S2010}). Marginal worthwhile changes refer to
an \textquotedblleft one step more" change, within the current period.
\item[(b)] \textbf{Worthwhile to change rather than to stay payoff}. It is, in the
current period $k+1$,
\[
\Delta _{e_{k},\xi_{k+1}}(x^{k},x^{k+1})=M_{e_{k}}(x^{k},x^{k+1})-\xi_{k+1}R_{e_{k}}(x^{k},x^{k+1}).
\]
Then, it is worthwhile to change from $x^{k}$ to $x^{k+1}$ iff $\ \Delta _{e_{k},\xi _{k+1}}(x^{k},x^{k+1})\geq 0$. Let us note, for simplification, $x^{k}=x$, $x^{k+1}=y,e_{k}=e,$ and $\xi
_{k+1}=\xi >0.$

\item[(c)] \textbf{Advantages to change}. Let us consider the following VR concepts. The VR approach considers three levels: the
case of an agent, an organization and interrelated agents.
\begin{itemize}
\item[i)] For an agent, let $A_{e}=A_{e}(x,y)\in \mathbb{R}$ be his advantage to change from 
$x$ to $y$. If $A_{e}(x,y)\geq 0,$ this agent has an advantage to change
from $x$ to $y.$ If $A_{e}(x,y)\leq 0$ he has a disadvantage to change from $%
x$ to $y$, i.e., a loss to change\ $F_{e}(x,y)=-A_{e}(x,y)\geq 0.$ Then,
advantages to change and loss functions are opposite. For example, in the
separable case, if actions $x$ and $y$ generate the to be improved payoffs $g(x),g(y)\in \mathbb{R}$, there is an advantage to change from doing $x$ to do $y$ if $%
g(y)\geq g(x)$, i.e., $A_{e}(x,y)=g(y)-g(x)\geq 0$. There is a loss to
change if $g(y)\leq g(x),$ i.e., 
\[
F_{e}(x,y)=-A_{e}(x,y)=g(x)-g(y)\geq 0.
\]
Let $f(x)$ be the to be decreased unsatisfied need of an agent, defined by $f(x)=\overline{g}
-g(x)\geq 0$, where $\overline{g}=\sup \left\{ g(z),z\in X\right\} <\infty $
is the highest feasible payoff of the agent. Then, after having done action $%
x$, this agent will have an advantage to change from $x$ to $y$ if $f(x)-f(y)\geq 0$, i.e., $A_{e}(x,y)\geq 0$. This means that his unsatisfied need will decrease, $f(y)\leq f(x)$;

\item[ii)] For a hierarchical organization with a leader and several followers,
where $x=(x^{l},x^{j},j\in J)$ and $y=(y^{l},y^{j},j\in J)$ are two
collective actions, let $g^{l}(x),$ $g^{l}(y)$ and $g^{j}(x),g^{j}(y)$, $%
j\in J$ be the \textquotedblleft to be increased" payoffs of the leader and all the followers
at $x$ and $y$. Then, advantages to change of the leader and all the
followers are 
\[
A^{l}(x,y)=g^{l}(y)-g^{l}(x)\geq 0,\ \ \mbox{and} \ \ A^{J}(x,y)=\Sigma_{j\in J}\left[ g^{j}(y)-g^{j}(x)\right] \geq 0,
\]
when they are nonnegative. In the other hand they represent losses to change $A^{l}(x,y)\leq 0$ and $
A^{J}(x,y)\leq 0$ when they are negative or zero. Then, the joint advantage
to change of the organization is
\[
A_{e}(x,y)=\varepsilon
A^{l}(x,y)+A^{J}(x,y),
\]
 where $\varepsilon >0$ is the weight allowed to the
leader payoff;
\item[iii)] The case of interrelated agents (games) has be examined elsewhere.
\end{itemize}
\item[(d)]\textbf{Inconvenients to change}. For an agent, let 
\begin{equation}\label{eq:IC1}
I_{e}(x,y):=C_{e}(x,y)-C_{e}(x,x)\geq 0,
\end{equation}
be his inconvenients to change. It
refers to the difference between his costs $C_{e}(x,y)$ to be able to change
from having the capability to do action $x$ one time more to having the
capability to do action $y$, and his costs to be able to stay, $%
C_{e}(x,x)\geq 0.$ For an organization with a leader $l$ and a given subset
of followers $J$, the costs to be able to change for all the members of the
organization are $C_{e}(x,y)=C_{e}^{l}(x,y)+\Sigma_{j\in
J}C_{e}^{j}(x,y)\in \mathbb{R}_{+}$. Then, the inconvenients to change of the
organization are the sum of the inconvenients to change of all members of
this organization.

\item[(e)]\textbf{Aspiration point. }Given a worthwhile transition $x^{k+1}\in
W_{e_{k},\xi_{k+1}}(x^{k})$, $k\in \mathbb{N}$, $x^{\ast }\in X$ is a strong aspiration
point if $x^{\ast }\in W_{e,_{k}\xi _{k+1}}(x^{k})$,  $k\in \mathbb{N}$. This means
that, starting from any position of the transition, it is worthwhile to
directly reach this aspiration point. Aspiration points are weak if it
exists $k_{0}\in \mathbb{N}$ such that $\ x^{\ast }\in W_{e_{k},\xi _{k+1}}(x^{k})$, $k\geq k_{0}$.

\item[(f)]\textbf{Stationary trap. }Given $e_{\ast }\in E$ and $\xi _{\ast }>0,$ $%
x^{\ast }\in X$\textbf{\ }is a stationary trap if $W_{e_{\ast },\xi _{\ast
}}(x^{\ast })=\left\{ x^{\ast }\right\} $. This means that $\Delta _{e_{\ast
},\xi _{\ast }}(x^{\ast },y)<0$ for all $y\neq x^{\ast }.$

\item[(g)]\textbf{Variational trap}. It is a point $x^{\ast }$ such that, starting
from a given intial point $x^{0}\in X,$ it exists a path of worthwhile
changes $x^{k+1}\in W_{e_{k},\xi _{k+1}}(x^{k})$ which ends in $x^{\ast },$
i.e., such that, being there, it is not worthwhile to move again, i.e., $%
W_{e_{\ast },\xi _{\ast }}(x^{\ast })=\left\{ x^{\ast }\right\} $. A
variational trap is both an aspiration point and a stationary trap.

\item[(h)]\textbf{Variatonal rationality problem. }Starting from $x^{0}\in X,$ find
when a given worthwhile transition $x^{k+1}\in W_{e_{k},\xi
_{k+1}}(x^{k})$,  $k\in \mathbb{N}$, converges to a variational trap $x^{\ast }\in X$. The
sequence of satisficing worthwhile to change ratio $\left\{\xi_{k+1}>0\right\}$ can be given, ex ante, or chosen, in an adaptive way,
each step.

\item[(i)]\textbf{An habituation/routinization process:} It is such that, step by
step, gradually, the agent carries out a more and more similar action. When
a worthwhile to change process converges to a variational trap, this
variational formulation offers a model of trap as the end point of a path of
worthwhile temporary stays and changes.
\end{enumerate}

\subsection{From behavioral sciences notations to mathematics}\label{sec:2.3}
In what follows, assuming that $f$ and $h$ are given as in the introduction of this paper, advantages to change for the leader and all followers are given by: $A_{e}^{l}(x,y)=-f(x,y)\geq 0$ and $A_{e}^{J}(x,y)=-h(x,y)\geq 0$, while their losses to change are $F_{e}^{l}(x,y)=f(x,y)%
\geq 0$ and $F_{e}^{J}(x,y)=h(x,y)\geq 0$ when they are non negative. The
weighted advantages to change for the organization and its weighted losses to change are, respectively, given by:
\begin{equation}\label{eq:ACLC1}
A_{e}(x,y):=-\left[\varepsilon f(x,y)+h(x,y)\right],\qquad F_{e}(x,y):=\varepsilon f(x,y)+h(x,y).
\end{equation}
We consider a simple and linear motivation and
resistance to change structure. In this case the utility of advantages to
change and the disutility of inconvenients to change are identical to
advantages and inconvenients to change, $U_{e}\left[ A_{e}(x,y)\right]
=A_{e}(x,y)$ and $D_{e}\left[ I_{e}(x,y)\right] =I_{e}(x,y)$. Then, the
worthwhile to change payoff of the organization is $\Delta _{e,\lambda
}(x,y)=A_{e}(x,y)-\lambda I_{e}(x,y)=-L_{e,\lambda }(x,y),$ where $%
L_{e,\lambda }(x,y)=F_{e}(x,y)+\lambda I_{e}(x,y).$

Next, we recall some definitions and results associated to the proximal distance and induced proximal distance,  useful in the remainder of the paper which have been handled in~\cite[Definition 2.1 and 2.2]{AT2006} and \cite{BD2010}. 

\begin{definition} A function $d :\mathbb{R}^n \times \mathbb{R}^n \rightarrow  \mathbb{R}_{+} \cup \{\infty \}$ is called a proximal distance with respect to a closed  nonempty convex set $S \subset \mathbb{R}^n$ iff for every fixed $y \in \mbox{\textnormal{int}}\, S$, the following properties hold:
\begin{description}
\item[(i)] $d(\cdot ,y)$ is a proper, lsc convex function and $C^1$ on $\mbox{\textnormal{int}}\, S;$
\item[(ii)] $\mbox{\textnormal{dom}}\, d(\cdot ,y) \subset S,$ and 
$\mbox{\textnormal{dom}}\, \partial _1 d(\cdot ,y) = \mbox{\textnormal{int}}\, S$, where 
$\partial _1 d(\cdot ,y)$ denotes the classical subgradient map of the function $d(\cdot ,y)$
with respect to the first variable.
\end{description}
\end{definition}
The family of functions satisfying this definition is denoted here by $\mathcal{D} (S)$.

Next step is to associate each given $d \in \mathcal{D} (S)$ with a corresponding proximal distance satisfying some desirable properties.
\begin{definition}\label{defH}
Given $d \in \mathcal{D} (S)$. Let $D :\mathbb{R}^n \times  \mathbb{R}^n \rightarrow  \mathbb{R}_{+} \cup \{ \infty \}$ be 
a function such that $\mbox{\textnormal{int}}\, S \times \mbox{\textnormal{int}}\, S \subseteq \mbox{\textnormal{dom}}\, D$.
$D$ is called the induced proximal distance to $d$ iff the following properties hold:
\begin{description}
\item [(H1)] For every $x \in \mbox{\textnormal{int}}\, S$,  $D(x,\cdot )$ is continuous on $\mbox{\textnormal{int}}\, S;$ 
\item [(H2)] $D(x,x) = 0$ for all  $x \in \mbox{\textnormal{int}}\, S;$
\item [(H3)] For all $x \in S$ and $\alpha  \in \mathbb{R}$, the set $\{y \in \mbox{\textnormal{int}}\, S : D(x,y) \leq \alpha \}$ is bounded;
\item [(H4)] For every $x,y \in \mbox{\textnormal{int}}\, S$, it holds that 
 $$\langle z - x, \nabla _1 d(x,y) \rangle \leq D(z,y) - D(z,x) - \gamma D(x,y)$$
for all $z \in S$ and some fixed $\gamma > 0;$
\item [(H5)] If $\{y^k\} \subset \mbox{\textnormal{int}}\, S$ and $y^k \rightarrow y \in S$, then $D(y,y^k) \rightarrow 0;$
\item [(H6)] Let $z \in S$ and $y \in \mbox{\textnormal{int}}\, S$, and take $w:=\alpha z + (1 - \alpha )y$, with $\alpha \in (0,1)$. Then 
\[
D(z,w) + D(w,y) \leq D(z,y). 
\]
\item[(H7)] If $\{x^k\},\{y^k\} \subset \mbox{\textnormal{int}}\, S$ are sequence such that $\{x^k\}$ converges to $x$ and $\{y^k\}$ 
converges to $y$, with $x \not= y$, then $\liminf \limits_{k} D(x^k, y^k) > 0.$
\end{description}
\end{definition}
{\it Remark 2.1.} 
The conditions {\bf{H1, H2, H3, H5}} and {\bf H7} on generalized distances refer to
more technical assumptions. The conditions {\bf H4} and {\bf H6} are related to a weak
form of the triangular inequality. The triangular inequality is a standard
assumption in the VR approach, (see \cite{S2009,S2010}, for a strong
justification).

We denote by $\mathcal{F} (S)$ the set of pairs $(d,D)$ of proximal and induced
proximal distances generalized that satisfies the conditions of Definition \ref{defH}, and we say that $(d,D)$ is a proximal generalized pair associated to $S$. In~\cite[Section 3]{AT2006} and ~\cite[Section 4]{BD2010} the authors give several examples of proximal distances, for instance; Bregman distances, a double regularization, or a second-order homogeneous proximal distances.   

{\it Remark 2.2.}
If $x=y$, it follows from property ${\bf H4}$ that $ \nabla _1 d(x,y)=0.$

Before stating the method, we recall two important facts regarding proximal distances generalized verifying {\bf H6} and {\bf H7}. The proofs of the following two propositions can be found in~\cite{BD2010}.
\begin{proposition}\label{lem-1} Assume that $(d,D)$ verifies {\bf H6} and {\bf H7}. If $\{x^k\} \subset S$ and $\{y^k\} \subset \emph{int}(S)$
 are sequences such that
 $$\lim_{k\rightarrow \infty}~D(x^k, y^k)=0,$$
 and one of the sequences $(\{x^k\}~\mbox{or}~ \{y^k\})$ converges, then the other one also converges to the same limit.
\end{proposition}
\begin{proposition}\label{lem2} Assume that $(d,D)$ verifies {\bf H6} and {\bf H7}. If $\{x^k\} \subset S$ and $\{y^k\} \subset \emph{int}(S)$
 are sequences such that
 $$\lim_{k\rightarrow \infty}~D(x^k, y^k)=0,$$
 and that one of the sequences $(\{x^k\}~\mbox{or}~ \{y^k\})$ is bounded. Then the following hold:
 \begin{description}
 \item[(a)] The other sequence is also bounded.
 \item[(b)] $\lim_{k\rightarrow \infty}~(x^k- y^k)=0$.
 \end{description}
\end{proposition}

\subsubsection*{Costs to be able to change} In this paper costs to be able to
change from the current profile of efforts $x=x^{k+1}$ to the future
profile $y$, given the experience $e=x^{k}$ of the agents,  are given by
\begin{equation}\label{eq:C1}
C_{e}(x,y)=\langle c_{e}(x),y-x\rangle+C_{e}(x,x),
\end{equation} 
where  $C_{e}(x,x)$ is the cost to be
able to stay,  i.e., cost to be able to repeat the same
current effort levels $x$ and $\langle c_{e}(x),y-x\rangle$ is the marginal cost to be able to change which are costs to be able to increase or decrease the
current effort levels, from $x$ to $y-x$. Note that, from  $\eqref{eq:C1}$,  the inconvenient to change, defined in $\eqref{eq:IC1}$, reduces
\begin{equation}\label{eq:IC2}
I_{e}(x,y)=\langle c_{e}(x),y-x\rangle.
\end{equation}
Given the generalized distance $%
d(x^{k},x)$ from the old \ profile of effort levels $x^{k}$ to the current
one $x=x^{k+1},$ the marginal costs to be able to change are $%
c_{k}(x)=\nabla d_{2}(x^{k},x)$. We have reversed the mathematical notation
from $\nabla d_{1}(x,x^{k})$ to $\nabla d_{2}(x^{k},x)$ to emphasize that \
the move is from $x^{k}$ to $x=x^{k+1}$. Generalized distances are not
symmetric, as required. They have regularity properties which are natural
for costs to be able to change. Furthermore, $C_{e}(x,x)\geq 0$ are usually
different from zero.

\section{The standards assumptions for equilibrium bifunctions}

Now, we pre-sent our basic assumptions associated to a given bifunction $\psi:K\times K\rightarrow \mathbb{R}$ and a proximal generalized pair  $(d,D)$. We assume that $K \subset \mbox{int(dom}\; d(\cdot,y))$ for all $y\in \mbox{int}(S)$.  For each $x,y\in K$ given,

\begin{description}
\item[$(L1)$]$\psi(x, x)=0$;
\item[$(L2)$]$\psi(\cdot, y) : K \longrightarrow  \mathbb{R}$ is upper semicontinuous;
\item[$(L3)$] $\psi (x, \cdot) : K \longrightarrow \mathbb{R}$ is convex and lower semicontinuous.
\end{description}
In addition to the previous hypotheses, we require the following properties for $g$:

\begin{description}
\item[$(L4)$] $\psi(x,y)+ \psi(y,x)\leq 0$;
\item[$(L5)$] For any sequence $\left\{y^n\right\}\subset K$ with 
$\displaystyle \lim_{n\rightarrow \infty}\|y^n\| =\infty$,
there exist $u \in K$ and $n_0 \in \mathbb{N}$ such that
\[
\psi(y^n,u)\leq 0,\qquad n \geq n_0.
\]
\end{description}

The behavioral content of the standard assumptions, 
can be described as follows:

\begin{itemize}
\item[i)]  L1 represents that the advantages (or losses) to change from $x$ to $y
$ are zero for any stay;
\item[ii)] L2 and L3 refer to regularity assumptions which are natural assumptions for advantages to change function in the VR
approach;
\item[iii)] L4 and two variants mean that if there is an advantage to change from $x
$ to $y$, there is a disadvantage to change from $y$ to $x$ (a no regret condition);
\item[iv)] L5 means that for any unbounded sequence of actions $\left\{
y^{n}\right\} $, there exists an aspiration point $u$ where the agent wants
to go, starting from any position $y^{n}$, $n\geq n_{0}$. An aspiration points
represents a major concept in the VR approach which characterizes a
variational trap as, both, an aspiration point (a desirability condition)
and a stationary trap (a stability condition).
\end{itemize}

{\it Remark 3.1.} 
Let us recall that Iusem and Sosa in~{\cite{IS2010}} presented $(L1)-(L5)$ hypotheses, in case where $d$ is an euclidean distance.  Burachik and Kassay in~{\cite{BK2012}} assumed these hypotheses by considering Bregman distances. 

\begin{theorem} 
Assume that $\psi$   satisfies  Conditions $(L1)-(L3)$ and $K$ is compact. Then, $S(\psi,K)$ is nonempty.
\end{theorem}
\begin{proof}
The proof is presented by Ky Fan in~\cite{KF1972}.  
\end{proof} 
\begin{theorem}\label{teo-1}  Assume that $\psi$ satisfies $(L1)$-$(L4)$ and $(L5)$. Then, $S(\psi,K)$ is nonempty.
\end{theorem}
\noindent
\begin{proof}
The proof is presented by Iusem et al. in~\cite[Theorem 4.3]{IKS2009}. 
\end{proof} 

\section{A generalized proximal distance as regularization}

Let $K \subset \mathbb{R}^n$ be a closed and convex set  and $f, h:K \times K \longrightarrow  \mathbb{R}$ such that conditions $(L1)-(L3)$ hold. Take  a  generalized proximal distance  $d$ with respect to $K$ 
and fix $\lambda, \varepsilon  > 0$, $\bar{x}\in \mbox{int}(S) $. Consider the bifunction 
$\bar{f}:K\times K\rightarrow \mathbb{R}$  defined by:
\begin{equation}\label{eq-6}
\bar{f}(x,y)= \varepsilon f(x,y)+ h(x,y)+\frac{1}{\lambda} \langle \nabla_{1} d(x,\bar{x}),y-x\rangle,
\end{equation}
where $\nabla_{1}d(x,\bar{x})$ denote the gradient of $d(.,\bar{x} )$ at $x$. We denote by $S(\bar{f},K)$ the solution set of the  equilibrium problem  associated with  $\bar{f}$. 

Usually, mathematicians consider ``to be decreased" payoffs (costs and
unsatisfied needs minimization problems). Behavioral Sciences consider, most
of the time, ``to be increased" payoffs (maximization, and more generally,
improvements of profits and utilities,$\ldots$). The VR variational rationality
approach follows this Behavioral Sciences tradition. As seen previouly, VR
advantages to change of the leader $l$ and the subset $J$ of follower are $%
A_{e}^{l}(x,y)=-f(x,y)\geq 0$ and $A_{e}^{J}(x,y)=-h(x,y)\geq 0,$ when such
advantages to change exist. This means that, using the mathematical
notation, $f(x,y)\leq 0$ and $h(x,y)\leq 0$ modelize advantages to change
and $f(x,y)\geq 0$ and $h(x,y)\geq 0$ modelize losses to change. So, in terms of the weighted
advantages to change of the organization and inconvenient to change given, respectively, in \eqref{eq:ACLC1} and \eqref{eq:IC2} (with $c_{e}(x)$ taken as being $\nabla_{1} d(x,\bar{x}$)), the bifunction in \eqref{eq-6} can be rewritten as follows:
\begin{equation}\label{eq:bifuncC}
\bar{f}(x,y)=-A_{e}(x,y)+\frac{1}{\lambda}I_{e}(x,\bar{x}).
\end{equation}

The proof of the next lemma is similar to that presented in~ \cite[Proposition 3.1]{BK2012}. 
\begin{lemma}\label{pro1} Fix $\bar{x}\in K $. Suppose that $f,h$ satisfies conditions $(L1)-(L4)$. Then $\bar{f}$ satisfies conditions $(L1)-(L4)$. Moreover, if for every sequence $\{y^n\} \subset  K$ such that $\lim_{n \rightarrow \infty} \|y^n\| = \infty$, we have
\[
{\bf L6} ~~~\liminf_{n \to \infty}\left[ g  (\bar{x}, y^n) + \lambda^{-1}\langle \nabla_{1} d(y^n,\bar{x}),y^n -\bar{x}\rangle \right ]> 0, 
\]
where $g=\varepsilon f+ h$, then $\bar{f}$ satisfies condition $(L5)$. 
\end{lemma}
\noindent
\begin{proof} It is clear that $\bar{f}$ satisfies condition $(L1)$. Since the  maps
\[
y\mapsto \langle \nabla_{1} d(x,\bar{x}),y-x\rangle
\]
is convex and continuous, and $x\mapsto \langle \nabla_{1} d(x,\bar{x}),y-x\rangle$ continuous at $x \in K$, it follows that $\bar{f}$ satisfies Conditions $(L2)-(L3)$. We claim now that $\bar{f}$ satisfies condition $(L4)$. Indeed,  from monotonicity of $g$ and $\nabla_{1} d(\cdot, z)$, we get 
\begin{align*}
\bar{f}(x,y)+\bar{f}(y,x) & = g(x,y)+g (y,x)-\lambda^{-1} \langle \nabla_{1} d(y,\bar{x})- \nabla_{1} d(x,\bar{x}),y-x\rangle \leq 0.
\end{align*}
Now, we show that $\bar{f}$ satisfies condition $(L5)$. From $(L4)$  and \eqref{eq-6}, we have
\begin{align*}
\bar{f}(y^n,\bar{x}) & = g(y^n,\bar{x})+ \lambda^{-1} \langle \nabla_{1} d(y^n,\bar{x}),\bar{x}-y^n\rangle \\
& \leq - [g (\bar{x},y^n)+\lambda^{-1} \langle  \nabla_{1} d(y^n,\bar{x}),y^n-\bar{x}\rangle].
\end{align*}
From ${\bf L6}$ there exists $n_0$ such that the expression in parentheses is nonnegative for all $n \geq n_0$. This implies that Condition $(L5)$ holds for $\bar{f}.$ 
\end{proof}

\begin{corollary} If $f, h$ satisfies conditions $(L1)-(L3)$, and assuming either
\begin{description}
\item[(i)] $K$ is bounded, or
\item[(ii)] $\nabla_{1} d(\cdot,\bar{x})$  is strongly monotone with modulus $\mu >0$, i.e.,
$$\langle \nabla_1 d(y,\bar{x})- \nabla_{1} d(x,\bar{x}),y-x\rangle \geq \mu \|y-x\|^2,\qquad x,y\in  K;$$
\end{description}
then $\bar{f}$ satisfies Condition $(L5)$.
\end{corollary}
\begin{proof}
 Using Lemma~\ref{pro1} it is enough to check that Condition ${\bf L6}$ holds under {\bf $(i)$} or {\bf $(ii)$}. Condition {\bf $(i)$} trivially implies Condition ${\bf L6}$. Hence it is enough to prove that {\bf $(ii)$} implies Condition ${\bf L6}$. We claim that $\mbox{Dom}\partial g(\bar{x},\cdot) \cap K \neq \emptyset$, where $g=\varepsilon f+ h$. Indeed, the subdifferential of a proper, lower semicontinuous and convex function is maximal monotone. If we extend the function $g(\bar{x},\cdot)$ to the whole space $\mathbb{R}^n$ by defining it as $ \infty$ outside $K$, then we have that $\partial g(\bar{x},\cdot)$ is maximal monotone. Thus $\mbox{Dom} g(\bar{x} , \cdot)$ should be nonempty. Since $\partial g(\bar{x},z) = \emptyset$ for every $z \notin  K$, it follows that $\mbox{Dom}\partial g(\bar{x} ,\cdot) \cap K \neq \emptyset$ must holds. Hence the claim is true and there exists $v \in \partial g(\bar{x},\bar{x} )$. Take a sequence $\{y^n\} \subset K$ such that $\|y^n\| \rightarrow \infty$, therefore we can write the subgradient inequality
$$g(\bar{x},y^n)\geq g(\bar{x},\bar{x})+ \langle v, y^n-\bar{x}\rangle \geq -\|v\|\|y^n-\bar{x}\|.$$
Altogether, we have
\begin{align*}
\liminf_{n \to \infty}\left[ g (\bar{x}, y^n) + \lambda^{-1}\langle \nabla_{1} d(y^n,\bar{x}),y^n -\bar{x}\rangle \right ] =& \\
\liminf_{n \to \infty}\left[ g(\bar{x}, y^n)+ \lambda^{-1}\langle \nabla_1 d(\bar{x},\bar{x})- \nabla_{1} d(y^n,\bar{x}),\bar{x}-y^n\rangle \right ] \geq \\
\liminf_{n \to \infty}\left[  -\|v\|\|y^n-\bar{x}\|+ \mu \lambda^{-1} \|y^n-\bar{x}\|^2 \right ] =\infty,
\end{align*}
and Condition ${\bf L6}$ is established. 
\end{proof}

Next result establishes the existence and uniqueness of the solution of equilibrium problem associated with $\bar{f}$. The proof is similar to presented in~\cite[Corollary 3.2]{BK2012}.
\begin{theorem}\label{teo-2} 
Under assumptions of Lemma~\ref{pro1}. The following assertions hold:
\begin{description}
\item[(i)] If Condition ${\bf L6}$ holds, then $S(\bar{f},K)$ is not empty;
\item[(ii)] $\nabla_{1} d(\cdot,\bar{x})$  is strictly monotone, then $S(\bar{f},K)$ has at most one element.
\end{description}
Altogether, if Condition ${\bf L6}$ holds and $\nabla_{1} d(\cdot,\bar{x})$ is strictly monotone, then $S(\bar{f},K)$ has a unique element. 
\end{theorem}
\noindent
\begin{proof} From Lemma~\ref{pro1} we have that $\bar{f}$ satisfies Condition $(L5)$. Using \cite[Theorem 4.3]{IKS2009}, we obtain that $S(\bar{f},K)$ is not empty. For proving ${\bf (ii)}$, assume that both $x_1, x_2\in S(\bar{f},K)$. Hence
\begin{equation*} \label{eq-7}
0\leq \bar{f}(x_1,x_2)=g(x_1,x_2)+\lambda^{-1} \langle \nabla_{1}d(x_1,\bar{x}),x_2-x_1\rangle,
\end{equation*}
and
\begin{equation*}\label{eq-8}
0\leq  \bar{f}(x_2,x_1)=g(x_2,x_1)+\lambda^{-1} \langle \nabla_{1}d(x_2,\bar{x}),x_1-x_2\rangle.
\end{equation*}
Adding last two inequalities, we get:
\begin{equation}\label{eq-9}
0\leq \bar{f}(x_1,x_2)+\bar{f}(x_1,x_2)\leq -\lambda^{-1}\langle \nabla_{1}d(x_1,\bar{x})-\nabla_{1}d(x_2,\bar{x}),x_1-x_2 \rangle.  
\end{equation}
Using the fact that $ \nabla_{1}d(\cdot,\bar{x})$ is strictly monotone, it follows from \eqref{eq-9} that $x_1=x_2$ as asserted. The last statement is a direct combination of ${\bf (i)}$ and ${\bf (ii)}$. 
\end{proof}

\section{A generalized proximal distance as  regularization method for solving bilevel equilibrium problems}\label{s3}  

From now on, we assume that $\nabla_{1} d(\cdot,\bar{x})$ is strictly monotone, $K \subset \mbox{int(dom}\; d(\cdot,y)) ~\mbox{for all}~ y\in \mbox{int}(S)$ and that all the hypotheses of Lemma~\ref{pro1} hold. 

\subsection{Proximal point algorithm}
In this section, following some ideas presented by Attouch et al. in~{\cite{ACP2011}} and Chbani and Riahi in~{\cite {CR2013}},  we present an approach of the proximal point algorithm with generalized distances for bilevel equilibrium problems, where the convergence result is obtained for bifunctions monotone.

From regularized problem \eqref{eq-6} and existence and uniqueness of its solution (see Theorem~\ref{teo-2}), we construct the following algorithm for solving the bilevel pro-blem~\eqref{BEP}. Take $\{\epsilon_{k}\}$ and $\{\lambda_{k}\}$ two sequences of positive real numbers such that $\sum_{1}^{\infty}\lambda_k=\infty$, $\varepsilon_{k}\to \infty$ and $\lambda_k\geq \theta> 0$.   
Consider the bifunction 
\begin{equation}\label{eq-21}
f_k(x,y)=\varepsilon_k f(x,y)+  h(x,y)+ \frac{1}{\lambda_k} \langle \nabla_{1} d(x,x^k),y-x\rangle,  \qquad x,y\in K,
\end{equation}
where $f,h$ satisfies conditions $(L1)-(L4)$
\begin{alg}\label{alg1} $\ $\\

\noindent
{\sc Initialization.} Choose an initial point $x^{0}\in K$;\\
{\sc Iterative Step.} Given $x^{k}$, take as next iterate $x^{k+1}\in K$ such that: 
\begin{equation}\label{mpp}
x^{k+1}\in \mbox{S}(f_k,K).
\end{equation}
{\sc Stopping criterion.} Given $x^{k}$,  if $x^{k+1}=x^{k}$ and $x^{k}\in \mbox{S}(f,K)$, STOP.\\

\end{alg}
{\it Remark 5.1.}
$\ $
\begin{itemize}
\item[(a)] 
Notice that if $h\equiv 0$ in \eqref{eq-21} it is sufficient to require, as a stopping criterion for the \emph{Algorithm~\ref{alg1}}, that $x^{k +1} = x^k$;
\item[(b)] If $\{x^k\}$ terminates after a finite number of iterations, then it terminates at a solution of \eqref{BEP}. Indeed, take $k$ such that $x^{k+1}=x^k$ and $x^k\in\mbox{S}(f,K)$. From definition of $x^{k+1}$ and $f_k$, and since $\nabla_1 d(x^{k+1},x^k)=0$, we obtain:
\begin{equation}\label{eq:2014-1}
\varepsilon_k f(x^{k+1}, y)+h(x^{k+1},y)\geq 0, \qquad  y\in K.
\end{equation}
Now, using that $x^{k+1}=x^k\in \mbox{S}(f,K)$ and $f$ is monotone, it follows that $f(y, x^{k+1})\leq 0$, for all $y\in K$. Hence, $f(x^{k+1}, y)=0$, for all $y\in \mbox{S}(f,K)$ and the statement follows from \eqref{eq:2014-1} for considering that $\varepsilon_k>0$;

\item[(c)] In term of the variational rationality approach, condition \eqref{mpp} is equivalent, each current period $k+1$, to the existence of the variational trap $x^{k+1}$. Indeed, combining definition of $f_{k}$ in \eqref{eq-21} with definition of $S(f_{k},K)$ and using \eqref{eq:bifuncC} with $e=e_{k}$, $\varepsilon=\varepsilon _{k}$, $\lambda=\lambda_{k}$, we obtain $x^{k+1}\in S(f_{k},K)$ if only if
\[
-f_{k}(x,y)=A_{e_{k}}(x,y)-\xi _{k+1}I_{e_{k}}(x,y)=\Delta _{e_{k,\xi_{k+1}}}(x,y),
\]
where $\xi _{k+1}=1/\lambda_{k}$. Then, it is not worthwhile to change from the current position $x=x^{k+1}$
to any position $y\in K$ iff 
\[
\Delta _{e_{k,\xi _{k+1}}}(x,y)\leq 0, \quad y\in K.
\]
This last condition defines a weak variational trap $x=x^{k+1}$ in the current
period $k+1.$
\end{itemize}


Next, we introduce a technical result on nonnegative sequences of real numbers that will be needed in the subsequence analysis. 
\begin{lemma}\label{lema0} Let $(\xi_k)$ and $(\gamma_k)$ be nonnegative sequences of real numbers satisfying:
\begin{description}
 \item[(a)] $\xi_{k+1}\leq \xi_k+ \gamma_k$;
 \item[(b)] $\sum_{k=0}^{\infty}~\gamma_k < \infty$.
\end{description}
Then, the sequence  $\{\xi_k\}$ converges.
\end{lemma}
\begin{proof}
The proof is presented by Polyak in~\cite[Lemma 9, p. 49]{P1987}.
\end{proof}

In the sequel, given a nonempty, closed convex set $\Omega\subset \mathbb{R}^n$, we denote by $\delta_{\Omega}$, $\mathcal{N}_{\Omega}$ and $\sigma_{\Omega}$, respectively, the indicator function, the normal cone and the support function associate to $\Omega$. Recall that:
{\small
\[
\delta_\Omega(x):=\left\{ 
\begin{array}{ll}
0, \quad \ \ \mbox{if} \ x\in \Omega;\\
\infty, \ \mbox{if} \ x\notin \Omega, \ 
\end{array}\right.
\mathcal{N}_\Omega(x):=\left\{ 
\begin{array}{ll}
\{q\in \mathbb{R}^n: \langle q , y-x \rangle \leq 0,  \ y\in \Omega \}, \ \mbox{if} \ x\in \Omega;\\
\emptyset,  \quad \mbox{if} \ x\notin \Omega,   
\end{array}\right. 
\]
}
and $\sigma_\Omega(x)=\sup_{y\in \Omega} \langle x , y\rangle$, $\delta^{\ast}_\Omega=\sigma_{\Omega}$ (elements presented by Rockafellar in~\cite[Theorem 13.2, p. 114]{R1970}), $\partial\delta_\Omega(x)=\mathcal{N}_\Omega(x)$  and, $y\in \mathcal{N}_\Omega(x)$ if and only if $\sigma_\Omega(y)=\langle y, x \rangle$, where $\delta_{\Omega}^{\ast}$ denote the conjugate function of $\delta_{\Omega}$.

Let us define the functions $f_z(y)=f(z,y)$ and $h_z(y)=h(z,y)$, $\forall y\in K$. 
\begin{lemma}\label{lem3-2} Assume that assumptions $(L1)$-$(L4)$ hold for $f$ and $h$, and take $z\in S(f,K)$, $w\in \partial\left( h_z+\delta_{S(f,K)}\right)(z)$ and $p\in \mathcal{N}_{S(f,K)}(z)$ such that $w-p \in\partial h_z(z).$ Then, 

\begin{eqnarray}
D(z,x^{k+1})+\frac{\lambda_k\varepsilon_k}{2} f(z,x^{k+1})&\leq& D(z,x^{k})-\gamma D(x^{k+1}, x^k)+\lambda_k\langle w, z-x^{k+1}\rangle\nonumber\\
&+&\frac{\lambda_k\varepsilon_k}{2}\left[ f_z^{\ast}(\frac{2p}{\varepsilon_k})-\sigma _{S(f,K)}(\frac{2p}{\varepsilon_k})\right].\nonumber 
\end{eqnarray}

\end{lemma}
\begin{proof} From \eqref{eq-21}, assumption ${\bf(H4)}$ and monotonicity of $f$ and $h$, we have 
{\small
\[
D(z,x^{k+1})+\frac{1}{2}\lambda_k\varepsilon_k f(z,x^{k+1})\leq D(z,x^{k})-\frac{1}{2}\lambda_k\varepsilon_k f(z, x^{k+1}) - \lambda_kh(z, x^{k+1}) -\gamma D(x^{k+1}, x^k).
\]
}
Since $w-p \in\partial h_z(z)$, it follows that $h(z, x^{k+1})\geq \langle w-p, x^{k+1}-z\rangle$.
Hence,
\begin{eqnarray}
D(z,x^{k+1})+\frac{1}{2}\lambda_k\varepsilon_k f(z,x^{k+1})&\leq& D(z,x^{k})-\frac{1}{2}\lambda_k\varepsilon_k f(z, x^{k+1})
\nonumber\\
&+& \lambda_k\langle w-p, z-x^{k+1}\rangle -\gamma D(x^{k+1}, x^k).\nonumber
\end{eqnarray}
Rewriting the last inequality, we obtain
{ 
\begin{eqnarray}
D(z,x^{k+1})+\frac{\lambda_k\varepsilon_k}{2} f(z,x^{k+1})&\leq& D(z,x^{k})-\gamma D(x^{k+1}, x^k)+\lambda_k\langle w, z-x^{k+1}\rangle\nonumber\\ 
&+& \frac{\lambda_k\varepsilon_k}{2}\left[\langle\frac{2p}{\varepsilon_k},x^{k+1} \rangle - f(z,x^{k+1})-\langle \frac{2p}{\varepsilon_k}, z \rangle - \delta_{S(f,K)}(z)\right].\nonumber   
\end{eqnarray}
}
Since $z\in S(f,K)$, $\delta_{S(f,K)}(z)=0$. Moreover, we have:
\begin{itemize}
\item[(i)] $\frac{2p}{\varepsilon_k}\in \partial\delta_{S(f,K)}(z)=\mathcal{N}_{S(f,K)}(z)$, so  $\delta_{S(f,K)}^{\ast}( \frac{2p}{\varepsilon_k}) =\sigma_{S(f,K)}(\frac{2p}{\varepsilon_k})$;
\item[(ii)] $\langle \frac{2p}{\varepsilon_k}, x^{k+1} \rangle- f_z(x^{k+1})\leq f_z^{\ast}(\frac{2p}{\varepsilon_k})$.
\end{itemize}
Therefore, the desired result follows by combining (i) and (ii) with last inequality.
\end{proof}
Under notations of the Lemma~\ref{lem3-2}, let us consider the following hypothesis:
\[
(\mathcal{H}): \qquad \sum_{k=1}^{\infty} \lambda_k\varepsilon_k\left[ f_z^{\ast}(q^k)-\sigma _{S(f,K)}(q^k)\right]< \infty,
\]
where $z\in S(f,K)$ and 
\[ 
q^k\in \mathcal{R}(\mathcal{N}_{S(f,K)}):=\{q\in \mathbb{R}^n:\exists~p\in S(f,K), \mbox{with} \ q \in N_{S(f,K)}(p)\}.
\]

{\it Remark 5.2.} 
 The assumption: $(\mathcal{H})$  is a geometric condition which is similar to what was introduced in linear setting in {\cite{ACP2011}} and appears in~{\cite{CR2013}}. In~{\cite{ACP2011}}, the authors showed that in case where $f_z(q)=\frac{1}{2} \mbox{dist}(q,K)^2$, we have $f_z^{\ast}(q)-\sigma _{K}(q)=\frac{1}{2}\|q\|^2$, for all $q\in \mathbb{R}^n$ and so 
\[
(\mathcal{H})\Longleftrightarrow \sum_{k=1}^{\infty} \frac{\lambda_k}{\varepsilon_k}< \infty.
\]

\begin{theorem}\label{teo-3} 
Assume that $f,h$ satisfies $(L1)$-$(L4)$, $\lambda_k\geq \theta> 0$ and $S(f,K)\neq \emptyset$. For all $x^0 \in K$, we have the following:
\begin{description}
\item[(i)] The sequence $\{x^k\}$ generated by \emph{Algorithm~\ref{alg1}} is well defined;
\item[(ii)] If $(\mathcal{H})$ holds, then for all $z \in S(h,S(f,K))$, the following hold:
 \begin{description}
 \item[(a)] exist $\lim_{k\rightarrow \infty } D(z, x^k)$;
 \item[(b)] The sequence $\{x^k\}$ is  bounded;
  \item[(c)] $\lim_{k\rightarrow \infty }D(x^{k+1},x^{k})=0$ and $\lim_{k\rightarrow \infty }(x^{k+1}-x^{k})=0$.
  \item[(d)] $\sum_{k=1}^{\infty}\lambda_k\varepsilon_k f(z,x^{k+1})< \infty$ 
 \end{description} 
\end{description}
\end{theorem}
\begin{proof}
Item ${\bf(i)}$  it follows from Theorem~\ref{teo-2}. For item ${\bf (ii)}$, take an arbitrary $z \in S(h,S(f,K))$, i.e., $h(z,y)\geq 0$ for all $y\in S(f,K)$.  Hence,
\[
z \in \mbox{argmin}_{y\in S(f,K)} h_z(y)\Longleftrightarrow 0\in \partial(h_z+\delta_{S(f,K)})(z).
\]
Taking $w=0$ in Lemma~\ref{lem3-2}, we have   
\begin{equation}\label{eqfej}
D(z,x^{k+1})\leq D(z,x^{k})+\frac{1}{2}\lambda_k\varepsilon_k\left[ f_z^{\ast}(\frac{2p}{\varepsilon_k})-\sigma _{S(f,K)}(\frac{2p}{\varepsilon_k})\right].
\end{equation}
Since assumption $(\mathcal{H})$ holds, by Lemma~\ref{lema0} there exists $\lim_{k\rightarrow \infty } D(z, x^k)$. In particular, $\{D(z, x^k)\}$ is a bounded set and by condition ${\bf(H3)}$, we conclude that the sequence $\{x^k\}$ is also bounded, it proves items (a) and (b). From Lemma~\ref{lem3-2}, we have
\[
\gamma D(x^{k+1}, x^k) \leq -D(z,x^{k+1})+ D(z,x^{k})+\frac{1}{2}\lambda_k\varepsilon_k\left[ f_z^{\ast}(\frac{2p}{\varepsilon_k})-\sigma _{S(f,K)}(\frac{2p}{\varepsilon_k})\right].
\]
Thus, $\lim_{k\rightarrow \infty }D(x^{k+1},x^{k})=0$. The Proposition~\ref{lem2} implies that 
$(x^{k+1}-x^{k})\rightarrow 0$ as $k\rightarrow \infty$, it complete the proof of (c). Again taking $w=0$ in Lemma~\ref{lem3-2}, we get  
\[
\frac{1}{2}\lambda_k\varepsilon_k f(z,x^{k+1})\leq D(z,x^{k})-D(z,x^{k+1}) +\frac{1}{2}\lambda_k\varepsilon_k\left[ f_z^{\ast}(\frac{2p}{\varepsilon_k})-\sigma _{S(f,K)}(\frac{2p}{\varepsilon_k})\right].
\]
Using item (a) and assumption $(\mathcal{H})$, we conclude the proof of (d).
\end{proof}

\begin{lemma} Assume that $f,h$ satisfies $(L1)$-$(L4)$ and $S(f,K)\neq \emptyset$. Let 
$\{x^{k}\}$ be the sequence generated by \emph{Algorithm~\ref{alg1}}. If $x^{k_j}\rightarrow \bar{x}$, then:
\begin{description}
 \item[(a)] $x^{k_j+1}\rightarrow \bar{x}$;
 \item[(b)] $h(y,\bar{x})\leq 0$ for all $y\in S(f,K)$;
  \item[(c)] $\bar{x}\in S(f,K)$. 
 \end{description} 
\begin{proof}
From triangular inequality, we have
\[
\|x^{k_j+1}- \bar{x} \|\leq \| x^{k_j+1}-x^{k_j}\|+\| x^{k_j}- \bar{x}\|.
\]
Taking the limit  as $k_j\rightarrow \infty$, we prove item (a). It follows of \eqref{eq-21} that 
\begin{equation}\label{eq3}
\varepsilon_k f(x^{k_j+1},y)+  h(x^{k_j+1},y)+ \frac{1}{\lambda_k} \langle \nabla_{1} d(x^{k_j+1},x^{k_j}),y-x^{k_j+1}\rangle\geq 0, \qquad \forall y\in K.
\end{equation}
Combinig the monotonicity of the bifunctions $f$ and $h$  with the last inequality, we get
\begin{eqnarray}
h(y,x^{k_j+1})&\leq& \frac{1}{\lambda_k} \langle \nabla_{1} d(x^{k_j+1},x^{k_j}),y-x^{k_j+1}\rangle \qquad \qquad \qquad \qquad \forall y\in S(f,K) \nonumber\\
&\leq& \frac{1}{\lambda_k}\left[ D(y,x^{k_j}) -D(y,x^{k_j+1})-\gamma D(x^{k_j+1}, x^{k_j})\right].\nonumber 
\end{eqnarray}
From last inequality, hypothesis $(L3)$ and using that $\lambda_k>0$, it follows 
\[
h(y,\bar{x})\leq \liminf h(y,x^{k_j+1})\leq 0, \qquad \forall\ y\in S(f,K).
\]
Since $\varepsilon_k\rightarrow \infty$, from inequality \eqref{eq3} and hyphotesis $(L2)$, we obtain   
\begin{eqnarray}
0&\leq& \limsup f(x^{k_j+1},y)+ \limsup \left[ \frac{1}{\varepsilon_k} h(x^{k_j+1},y)+ \frac{1}{\lambda_k\varepsilon_k } \langle \nabla_{1} d(x^{k_j+1},x^{k_j}),y-x^{k_j+1}\rangle\right]\nonumber\\ 
&\leq&
f(\bar{x},y), \nonumber
\end{eqnarray}
for all $y\in K$. Therefore, $\bar{x}\in S(f,K)$.
Now, using \eqref{eq3} and that $f(y,x^{k_j+1})\geq 0$, $y\in S(f,K)$, we get   
\[
0\leq h(x^{k_j+1},y)+ \frac{1}{\lambda_k} \langle \nabla_{1} d(x^{k_j+1},x^{k_j}),y-x^{k_j+1}\rangle, \qquad y\in S(f,K). 
\]
Thus,
\[
0 \leq h(x^{k_j+1},y)+ \frac{1}{\lambda_k}\left[ D(y,x^{k_j}) -D(y,x^{k_j+1})-\gamma D(x^{k_j+1}, x^{k_j})\right],\qquad y\in S(f,K). 
\]
Hence,
\[
0 \leq \limsup \left( h(x^{k_j+1},y)+\frac{ D(y,x^{k_j}) -D(y,x^{k_j+1})-\gamma D(x^{k_j+1}, x^{k_j})}{\lambda_k}  \right)  \leq h(\bar{x},y) 
\]
for all $y\in S(f,K)$.
\end{proof}

\end{lemma}

\subsection{Convergence of a worthwhile stays and changes transition to a
bilevel equilibrium}
$\ $
\begin{definition}\label{d1} A sequence $\left\{z^{k}\right\}\subset \mathbb{R}^n$ is said to be quasi-Fej$\acute{e}$r convergent to a set $U \neq \emptyset$ with respect to the proximal distance generalized  $(d,D)$,  if there exists a non-negative summable sequence $\{\gamma_k\}$ such that, for each $k\in\mathbb{N}$,
\[ 
D(z^{k+1},u)\leq D(z^{k},u)+\gamma _k,\qquad   u \in U.
\]   
\end{definition}
\noindent

Next result is important to establish the convergence of the sequence generated by Algorithm~\ref{alg1}.
\begin{lemma}\label{lem-f} If $\left\{z^{k}\right\} \subset \mathbb{R}^n$ is quasi-Fej$\acute{e}$r convergent to a set $U \neq \emptyset$ with respect to the proximal distance generalized  $(d,D)$, then $\left\{z^{k}\right\}$ is bounded. If a cluster point $z$ of $\left\{z^{k}\right\}$ belongs to $U$, then $\lim_{k \rightarrow \infty}z^k=z.$
\end{lemma}
\noindent
\begin{proof}
The proof is presented by Iusem et al. in~\cite{IST1994}.  
\end{proof}

Let us show that the whole sequence $\left\{x^{k}\right\}$ converges to a solution of \eqref{BEP}. 

\begin{theorem}\label{teo-4} Under assumptions of \emph{Theorem~\ref{teo-3}}. The whole sequence $\{x^{k}\}$, generated by \emph{Algorithm~\ref{alg1}}, converges to a solution of \eqref{BEP}. 

\end{theorem}
\noindent
\begin{proof}

From~\eqref{eqfej} and Lemma~\ref{lema0}, we obtain that $\{x^k\}$  is quasi-Fej$\acute{e}$r convergent to a set $S(h,S(f,K)) \neq \emptyset$ with respect to generalized proximal distance $(d,D)$. Moreover there exists a cluster point $\bar{x}$ of $\{x^k\}$ such that $\bar{x}\in S(h,S(f,K))$. Thus, by Lemma~\ref{lem-f}, we concluded the proof.
\end{proof}

\subsubsection{The stability and change dynamics of hierarchical
organizations}

From the viewpoint behavioral, we have the following findings:

\begin{itemize}
\item [a)] Our application shows that a succession of worthwhile
temporary stays and changes from a current stationary trap to the next one
converges to a bilevel equilibrium, for an organization which can change
(bargain), within some bounds, each period, its sharing rules $\varepsilon
_{k}>0,\xi _{k+1}=1/\lambda _{k}$ between the leader and workers VR
advantages to change payoffs and their VR inconvenients to change;

\item [b)] Setting bounded sharing rules (which can change within
bounds) allows convergence of the allocation of tasks to a stable
hierarchical weak variational trap. The striking point is that the formation
of habitual tasks can occur even in a non stationary environment. Each
period, bargaining over payoffs destabilizes the current weak variational
trap, and the process goes on, until reaching, at the end, a hierarchical
(bilevel) equilibrium. However, each period, the allocation of tasks becomes
more and more similar, ending in a routinized allocation of tasks.
\end{itemize}


\end{document}